\documentclass{amsart}

\usepackage[T1]{fontenc}
\usepackage{amscd,amssymb}
\usepackage{hyperref}

\allowdisplaybreaks[2]

\numberwithin{equation}{section}

\newtheorem{dummy}{dummy}[section]
\newtheorem{lemma}[dummy]{Lemma}
\newtheorem{theorem}[dummy]{Theorem}

\newtheorem{proposition}[dummy]{Proposition}

\theoremstyle{definition}
\newtheorem*{acknowledgments}{Acknowledgments}

\newcommand{\C}{\mathbb{C}}
\newcommand{\R}{\mathbb{R}}
\newcommand{\Z}{\mathbb{Z}}

\newcommand{\cL}{\mathcal{L}}
\newcommand{\cE}{\mathcal{E}}
\newcommand{\cO}{\mathcal{O}}
\newcommand{\cF}{\mathcal{F}}
\newcommand{\cI}{\mathcal{I}}
\newcommand{\cM}{\mathcal{M}}

\newcommand{\Aut}{\mathrm{Aut}}
\newcommand{\Sym}{\mathrm{Sym}}
\newcommand{\Pic}{\mathrm{Pic}}

\newcommand{\Hom}{\mathrm{Hom}}
\newcommand{\rk}{\mathrm{rk}}
\renewcommand{\deg}{\mathrm{deg}}

\newcommand{\si}{\sigma}
\newcommand{\Xsi}{X^\sigma}
\newcommand{\lra}{\longrightarrow}
\newcommand{\lmt}{\longmapsto}
\newcommand{\ov}[1]{\overline{#1}}
\newcommand{\os}[1]{\overline{\sigma^*{#1}}}
\newcommand{\CP}{\mathbb{C}\mathbf{P}}
\newcommand{\RP}{\mathbb{R}\mathbf{P}}
\newcommand{\ModC}{\mathcal{M}_{X}(r,d)}
\newcommand{\ModCodd}{\mathcal{M}_{X}(r,2e+1)}
\newcommand{\ModCeven}{\mathcal{M}_{X}(r,2e)}
\newcommand{\U}{\mathbf{U}}
\newcommand{\T}{\mathbf{T}}
\newcommand{\W}{\mathbf{W}}

\renewcommand{\phi}{\varphi}

\begin{document}

\title{Vector bundles over a {R}eal elliptic curve}
\author{Indranil Biswas}
\address{School of Mathematics, Tata Institute of Fundamental Research, Homi Bhabha Rd, Mumbai 400005, India.}
\email{indranil@math.tifr.res.in}
\author{Florent Schaffhauser}
\address{Departamento de Matem\'aticas, Universidad de Los Andes, Bogot\'a, Colombia.}
\email{florent@uniandes.edu.co}
\subjclass[2000]{14H52,14H60,14P25}

\keywords{Elliptic curves, vector bundles on curves, topology of real algebraic
varieties}

\date{\today}

\begin{abstract}
Given a geometrically irreducible smooth projective curve of genus $1$ defined over
the field of real numbers, and a pair of integers $r$ and $d$, we determine the isomorphism class
of the moduli space of semi-stable vector bundles of rank $r$ and degree $d$ on the
curve. When $r$ and $d$ are coprime, we describe the topology of the real locus and give a modular interpretation of its points. We also study, for arbitrary rank and degree, the moduli space of indecomposable vector bundles of rank $r$ and degree $d$, and determine its isomorphism class as a real algebraic variety.
\end{abstract}

\maketitle

\tableofcontents

\section{Introduction}

\subsection{Notation}
In this paper, a \textit{real elliptic curve} will be a triple $(X,x_0,\si)$ where 
$(X,x_0)$ is a complex elliptic curve (i.e.,\ a compact connected Riemann surface of 
genus $1$ with a marked point $x_0$) and $\si:X\lra X$ is an anti-holomorphic 
involution (also called a real structure). We do not assume that $x_0$ is fixed under 
$\si$. In particular, $\Xsi\,:=\,\mathrm{Fix}(\si)$ is allowed to be empty.

The g.c.d.\ of two integers $r$ and $d$ will be denoted by $r\wedge d$.

In the introduction, we omit the definitions of stability and semi-stability of vector 
bundles, as well as that of real and quaternionic structures; all these definitions 
will be recalled in Section \ref{semi-stable_bundles}.

\subsection{The case of genus zero}

Vector bundles over a real Riemann surface of genus $g\geq 2$ have been studied from various points of view in the past few years: moduli spaces of real and quaternionic vector bundles were introduced through gauge-theoretic techniques in \cite{BHH}, then related to the real points of the usual moduli variety in \cite{Sch_JSG}. In genus $0$, there are, up to isomorphism, only two possible real Riemann surfaces: the only compact Riemann surface of genus $0$ is the Riemann sphere $\CP^1$ and it can be endowed either with the real structure $[z_1:z_2]\lmt [\ov{z_1}:\ov{z_2}]$ or with the real structure $[z_1:z_2]\lmt [-\ov{z_2}:\ov{z_1}]$. The real locus of the first real structure is $\RP^1$ while the real locus of the second one is empty. Now, over $\CP^1$, two holomorphic line bundles are isomorphic if and only if they have the same degree and, by a theorem due to Grothendieck (\cite{Grot_P1}), any holomorphic vector bundle over the Riemann sphere is isomorphic to a direct sum of line bundles. So, over $\CP^1$, the only stable vector bundles are the line bundles, a semi-stable vector bundle is necessarily poly-stable and any vector bundle is isomorphic
to a direct sum of semi-stable vector bundles, distinguished by their respective slopes
and ranks. In particular, if $\cE$ is semi-stable of rank $r$ and degree $d$, then $r$ divides $d$ and $$\cE \simeq \cO(d/r) \oplus \cdots \oplus \cO(d/r),$$ where $\cO(1)$ is the positive-degree generator of the Picard group of $\CP^1$ and $\cO(k)$ is its $k$-th tensor power. This means that the moduli space of semi-stable vector bundles of rank $r$ and degree $d$ over $\CP^1$ is 
$$\mathcal{M}_{\CP^1}(r,d) = \left\{
\begin{array}{ccc}
\{\mathrm{pt}\} & \mathrm{if}& r\,|\,d,\\
\emptyset & \mathrm{if}& r\not|d.
\end{array}\right.$$ Assume now that a real structure $\si$ has been given on $\CP^1$. Then, if $\cL$ is a holomorphic line bundle of degree $d$ over $\CP^1$, it is isomorphic to its Galois conjugate $\os{\cL}$, since they have the same degree. This implies
that $\cL$ is either real or quaternionic. Moreover, this real (respectively,\
quaternionic) structure is unique up to real (respectively,\ quaternionic)
isomorphism (see Proposition \ref{self_conj_stable_bundles}). If the real structure
$\si$ has real points, then quaternionic bundles must have even rank. Thus, when
$\mathrm{Fix}(\si) \neq \emptyset$ in $\CP^1$, any line bundle, more generally any
direct sum of holomorphic line bundles, admits a canonical
real structure. Of course, given a real vector bundle of the
form $(\cL\oplus\cL,\tau\oplus\tau)$, where $\tau$ is a real structure on the line bundle $\cL$, one
can also construct the quaternionic structure
$\begin{pmatrix} 0 & -\tau \\ \tau & 0 \end{pmatrix}$ on $\cL\oplus\cL$.
Note that the real vector bundle $(\cL\oplus\cL \, ,\begin{pmatrix} 0 & \tau
\\ \tau & 0 \end{pmatrix})$ is isomorphic to $(\cL\oplus\cL,\tau\oplus\tau)$.
When $\CP^1$ is equipped with its real structure with no real points, a given
line bundle $\cL$ of degree $k$ is again necessarily self-conjugate, so it has to
be either real or quaternionic but now real line bundles must have even degree
and quaternionic line bundles must have odd degree (\cite{BHH}), so $\cL$ admits a
canonical real structure if $k$ is even and a canonical quaternionic structure if $k$
is odd. Consequently, when $\mathrm{Fix}(\si)=\emptyset$ in $\CP^1$, semi-stable
holomorphic vector bundles of rank $r$ and degree $d=rk$ over $\CP^1$ admit a canonical
real structure if $k$ is even and a canonical quaternionic structure if $k$ is odd.

\subsection{Description of the results}

The goal of the present paper is to analyze that same situation in the case of real 
Riemann surfaces of genus one. In particular, we completely identify the moduli space 
of semi-stable holomorphic vector bundles of rank $r$ and degree $d$ as a real 
algebraic variety (Theorem \ref{moduli_space_over_R} below). Our main references are 
the papers of Atiyah (\cite{Atiyah_elliptic_curves}) and Tu (\cite{Tu}). In what 
follows, we denote by $(X,x_0)$ a complex elliptic curve and by $\ModC$ the moduli 
space of semi-stable vector bundles of rank $r$ and degree $d$ over $X$, i.e.,\ the 
set of $S$-equivalence classes of semi-stable holomorphic vector bundles of rank $r$ 
and degree $d$ over $X$ (\cite{Seshadri}). Since $(X,x_0)$ is an elliptic curve, the 
results of Atiyah show that any holomorphic vector bundle on $X$ is again (as in 
genus $0$) a direct sum of semi-stable vector bundles (see Theorem \ref{Tu_obs}) but now 
there can be semi-stable vector bundles which are not poly-stable (see 
\eqref{ses_defining_F_h}) and also there can be stable vector
bundles of rank higher than $1$. 
Moreover, the moduli space $\ModC$ is a non-singular complex algebraic variety of 
dimension $h:=r\wedge d$. As a matter of fact, $\ModC$ 
is isomorphic, as a complex algebraic variety, to the $(r\wedge d)$-fold symmetric 
product $\Sym^{r\wedge d}(X)$ of the complex elliptic curve $X$ and it contains 
stable bundles if and only if $r\wedge d=1$, in which case all semi-stable bundles 
are in fact stable. Let now $\si:X\lra X$ be a real structure on $X$ (recall that the 
marked point $x_0$ is not assumed to be fixed under $\si$). Then the map
$\cE\lmt\os{\cE}$ induces a real structure, again 
denoted by $\si$, on $\ModC$, since it preserves the rank, the degree and the 
S-equivalence class of semi-stable vector bundles (\cite{Sch_JSG}). Our main result 
is then the following, to be proved in Section \ref{real_structure_on_mod_space}.

\begin{theorem}\label{moduli_space_over_R}
Let $h:=r\wedge d$. Then, as a real algebraic variety, $$\ModC \simeq_{\R} \left\{
\begin{array}{cl}
\Sym^h (X) & \mathrm{if}\ \Xsi\neq\emptyset,\\
\Sym^h (X) & \mathrm{if}\ \Xsi=\emptyset\ \mathrm{and}\ d/h\ \mathrm{is\ odd},\\
\Sym^h (\Pic^{\, 0}_X) & \mathrm{if}\ \Xsi=\emptyset\ \mathrm{and}\ d/h\ \mathrm{is\ even}.
\end{array}
\right.$$
\end{theorem}

We recall that $\Pic^{\,0}_X$ is isomorphic to $X$ over $\C$ (via the choice of 
$x_0$) but not over $\R$ when $\Xsi=\emptyset$ because $\Pic^{\,0}_X$ has the
real point corresponding to the trivial line bundle. In contrast, $\Pic^{\,1}_X$ is 
always isomorphic to $X$ over $\R$, as we shall recall in Section \ref{line_bundles}. 
For any $d\in\Z$, the real structure of $\Pic^{\,d}_X$ is induced by the map 
$\cL\lmt\os{\cL}$, while the real structure of the $h$-fold symmetric product 
$\Sym^h(Y)$ of a real variety $(Y,\si)$ is induced by that of $Y$ in the following 
way: $[y_1,\cdots,y_h]\lmt[\si(y_1),\cdots,\si(y_h)]$. Note that, if $r\wedge d=1$, 
then by Theorem \ref{moduli_space_over_R} we have $\ModC\simeq_{\R} X$ if 
$\Xsi\neq\emptyset$ or $d$ is odd, and $\ModC\simeq_{\R} \Pic^{\, 0}_X$ if 
$\Xsi=\emptyset$ and $d$ is even. This will eventually imply the following results on 
the topology and modular interpretation of the set of real points of $\ModC$, 
analogous to those of \cite{Sch_JSG} for real curves of genus $g\geq 2$ (see Section 
\ref{topology} for a proof of Theorem \ref{real_pts_coprime_case}; we point out that 
it will only be valid under the assumption that $r\wedge d=1$, in which case all 
semi-stable bundles are in fact stable, in particular a real point of $\ModC$ is given 
by either a real bundle or a quaternionic bundle, in an essentially 
unique way; see Proposition \ref{self_conj_stable_bundles}).

\begin{theorem}\label{real_pts_coprime_case}
Assume that $r\wedge d=1$.
\begin{enumerate}
\item If $\Xsi\neq\emptyset$, then $\ModC^\si\simeq \Xsi$ has either $1$ or $2$ connected components. Points in either component correspond to real isomorphism classes of real vector bundles of rank $r$ and degree $d$ over $(X,\si)$ and two such bundles $(\cE_1,\tau_1)$ and $(\cE_2,\tau_2)$ lie in the same connected component if and only if $w_1(\cE^{\tau_1})=w_1(\cE_2^{\tau_2})$.
\item If $\Xsi=\emptyset$ and $d=2e+1$, then $\ModCodd^\si\simeq \Xsi$ is empty.
\item If $\Xsi=\emptyset$ and $d=2e$, then $\ModCeven^\si\simeq (\Pic^{\, 0}_X)^\si$ has two connected components, one consisting of real isomorphism classes of real bundles, the other consisting of quaternionic isomorphism classes of quaternionic bundles. These two components become diffeomorphic under the operation of tensoring a given bundle by a quaternionic line bundle of degree $0$.
\end{enumerate} Moreover, in cases $\textstyle{(1)}$ and $\textstyle{(3)}$, each connected component of $\ModC^\si$ is diffeomorphic to $S^1$.
\end{theorem}

\noindent In particular, the formulae of \cite{LS} (see also \cite{Baird}), giving the mod $2$ Betti numbers of the connected components of $\ModC^\si$ when $r\wedge d=1$ are still valid for $g=1$ (in contrast, when $r\wedge d\neq1$, the formulae of \cite{LS} do not seem to be interpretable in any way since, over an elliptic curve, the dimension of $\ModC$ is $r\wedge d$, not $r^2(g-1)+1$).

In the third and final section of the paper, we investigate the properties of indecomposable vector bundles over real elliptic curves. Recall that a holomorphic vector bundle $\cE$ over a complex curve $X$ is said to be indecomposable if it is not isomorphic to a direct sum of non-trivial holomorphic bundles. When $X$ is of genus $1$, there exists a moduli variety $\cI_X(r,d)$ whose points are isomorphism classes of indecomposable vector bundles of rank $r$ and degree $d$: it was constructed by Atiyah in \cite{Atiyah_elliptic_curves} and revisited by Tu in \cite{Tu}, as will be recalled in Theorems \ref{Atiyah_indecomp_bdles} and \ref{rel_between_stable_and_indecomp}. We will  then see in Section \ref{indecomposable_bdles_over_real_elliptic_curves} that we can extend their approach to the case of real elliptic curves and obtain the following characterization of $\cI_X(r,d)$ as a real algebraic variety.

\begin{theorem}\label{indecomp_bdles_over_R}
Let $(X,x_0,\si)$ be a real elliptic curve. Let $\mathcal{I}_X(r,d)$ be the set of isomorphism classes of indecomposable vector bundles of rank $r$ and degree $d$ and let us set $h:=r\wedge d$, $r':=\frac{r}{h}$, $d':=\frac{d}{h}$. Then:
$$\cI_X(r,d) \simeq_\R \mathcal{M}_X(r',d') \simeq_{\R} \left\{ \begin{array}{cl}
X & \mathrm{if}\ \Xsi\neq \emptyset,\\
X & \mathrm{if}\ \Xsi= \emptyset\ \mathrm{and}\ d'\ \mathrm{is\ odd},\\
\Pic^{\,0}_X & \mathrm{if}\ \Xsi= \emptyset\ \mathrm{and}\ d'\ \mathrm{is\ even}.
\end{array}\right.
$$
\end{theorem}

\noindent By combining Theorems \ref{real_pts_coprime_case} and \ref{indecomp_bdles_over_R}, we obtain the following topological description of the set of real points of $\cI_X(r,d)$, valid even when $r\wedge d\neq 1$.

\begin{theorem}\label{real_pts_indecomp_bdles}
Denote by $\cI_X(r,d)^\si$ the fixed points of the real structure $\cE\lmt\os{\cE}$ in $\cI_X(r,d)$.
\begin{enumerate}
\item If $\Xsi\neq \emptyset$, then $\cI_X(r,d)^\si\simeq X^\si$ consists of real isomorphism classes of real and indecomposable vector bundles of rank $r$ and degree $d$. It has either one or two connected components, according to whether $\Xsi$ has one or two connected components, and these are distinguished by the Stiefel-Whitney classes of the real parts of the real bundles that they contain.
\item If $\Xsi=\emptyset$ and $\frac{d}{r\wedge d}=2e+1$, then $\cI_X(r,d)^\si\simeq \Xsi$ is empty.
\item If $\Xsi=\emptyset$ and $\frac{d}{r\wedge d}=2e$, then $\cI_X(r,d)^\si\simeq (\Pic^{\, 0}_X)^\si$ has two connected components, one consisting of real isomorphism classes of vector bundles which are both real and indecomposable and one consisting of quaternionic isomorphism classes of vector bundles which are both quaternionic and indecomposable. These two components become diffeomorphic under the operation of tensoring a given bundle by a quaternionic line bundle of degree $0$.
\end{enumerate} Moreover, in cases $\textstyle{(1)}$ and $\textstyle{(3)}$, each connected component of the set of real points of $\cI_X(r,d)$ is diffeomorphic to $S^1$.
\end{theorem}

\begin{acknowledgments}
The authors thank the Institute of Mathematical Sciences of the National University 
of Singapore, for hospitality while the work was carried out. The first author is supported by J. C. Bose Fellowship. The second 
author acknowledges the support from U.S. National Science Foundation grants DMS 
1107452, 1107263, 1107367 "RNMS: Geometric structures And Representation varieties" 
(the GEAR Network). Thanks also go to the referee for a careful reading of the paper and for suggesting the reference \cite{Baird}.
\end{acknowledgments}

\section{Moduli spaces of semi-stable vector bundles over an elliptic curve}\label{semi-stable_bundles}

\subsection{Real elliptic curves and their Picard varieties}\label{line_bundles}

The real points of Picard varieties of real algebraic curves have been studied for instance by Gross and Harris in \cite{GH}. We summarize here some of their results, specializing to the case of genus $1$ curves.

Let $X$ be a compact connected Riemann surface of genus $1$. To each point $x\in X$, there is associated a holomorphic line bundle $\cL(x)$, of degree $1$, whose holomorphic sections have a zero of order $1$ at $x$ and no other zeros or poles. Since $X$ is compact, the map $X\lra \Pic^{\,1}_X$ thus defined, called the Abel-Jacobi map, is injective. And since $X$ has genus $1$, it is also surjective. The choice of a point $x_0\in X$ defines an isomorphism $\Pic^{\,0}_X\overset{\simeq}{\lra} \Pic^{\,1}_X$, obtained by tensoring by $\cL(x_0)$. In particular, $X\simeq\Pic^{\,1}_X$ is isomorphic to $\Pic^{\,0}_X$ as a complex analytic manifold and inherits, moreover, a structure of Abelian group with $x_0$ as the neutral element.

If $\si:X\lra X$ is a real structure on $X$, the Picard variety $\Pic^{\,d}_X$, whose points represent isomorphism classes of holomorphic line bundles of degree $d$, has a canonical real structure, defined by $\cL\lmt\os{\cL}$ (observe that this anti-holomorphic involution, which we will still denote by $\si$, indeed preserves the degree). Since $\cL(\si(x)) \simeq \os{(\cL(x))}$, the Abel-Jacobi map $X\lra\Pic^{\,1}_X$ is defined over $\R$, meaning that it commutes to the real structures of $X$ and $\Pic^{\,1}_X$. We also call such a map a real map. If $X^\si\neq\emptyset$, we can choose $x_0\in X^\si$ and then $\cL(x_0)$ will satisfy $\os{\cL(x_0)} \simeq \cL(x_0)$ so the isomorphism $\Pic^{\,0}_X \overset{\simeq}{\lra} \Pic^{\,1}_X$ obtained by tensoring by $\cL(x_0)$ will also be defined over $\R$. More generally, by tensoring by a suitable power of $\cL(x_0)$, we obtain real isomorphisms $\Pic^{\,d}_X \simeq \Pic^{\,1}_X$ for any $d\in\Z$. If now $X^\si=\emptyset$, then we actually cannot choose $x_0$ in such a way that $\cL(\si(x_0)) \simeq \cL(x_0)$ (see \cite{GH} or Theorem \ref{GH_case_g_equal_1} below; the reason is that such a line bundle would be either real or quaternionic but, over a real curve of genus $1$ with no real points, real and quaternionic line bundles must have even degree) but we may consider the holomorphic line bundle of degree $2$ defined by the divisor $x_0+\si(x_0)$, call it $\cL$, say. Then $\os{\cL}\simeq \cL$ and, by tensoring by an appropriate tensor power of it, we have the following real isomorphisms $$\Pic^{\,d}_X \simeq_\R \left\{ \begin{array}{ccl}
\Pic^{\,1}_X & \mathrm{if} & d=2e+1,\\
\Pic^{\,0}_X & \mathrm{if} & d=2e.
\end{array}\right.$$ So, when the genus of $X$ is $1$, we have the following result.

\begin{theorem}\label{line_bdle_case}
Let $(X,x_0,\si)$ be a real elliptic curve.
\begin{enumerate}
\item If $X^\si\neq\emptyset$, then for all $d\in \Z$, $$\Pic^{\,d}_X\simeq_\R X.$$
\item If $X^\si = \emptyset$, then $$\Pic^{\,d}_X \simeq_\R \left\{ \begin{array}{ccl}
X & \mathrm{if} & d=2e+1,\\
\Pic^{\,0}_X & \mathrm{if} & d=2e.
\end{array}\right.$$ 
\end{enumerate}
\end{theorem}

\subsection{Semi-stable vector bundles}

Let $X$ be a compact connected Riemann surface of genus $g$ and recall that the slope of a non-zero holomorphic vector bundle $\cE$ on $X$ is by definition the ratio $\mu(\cE)=\deg(\cE)/\rk(\cE)$ of its degree by its rank. The vector bundle $\cE$ is called stable (respectively,\ semi-stable) if for any non-zero proper sub-bundle $\cF\subset \cE$, one has $\mu(\cF) < \mu(\cE)$ (respectively,\ $\mu(\cF) \leq \mu(\cE)$). By a theorem of Seshadri (\cite{Seshadri}), any semi-stable vector bundle $\cE$ of rank $r$ and degree $d$ admits a filtration whose successive quotients are stable bundles of the same slope, necessarily equal to $d/r$. Such a filtration, called a Jordan-H\"older filtration, is not unique but the graded objects associated to any two such filtrations are isomorphic. The isomorphism class thus defined is denoted by $\mathrm{gr}(\cE)$ and holomorphic vector bundles which are isomorphic to direct sums of stable vector bundles of equal slope are called poly-stable vector bundles. Moreover, two semi-stable vector bundles $\cE_1$ and $\cE_2$ are called $S$-equivalent if $\mathrm{gr}(\cE_1)= \mathrm{gr}(\cE_2)$ and Seshadri proved in \cite{Seshadri} that, when $g\geq 2$, the set of $S$-equivalence classes of semi-stable vector bundles of rank $r$ and degree $d$ admits a structure of complex projective variety of dimension $r^2(g-1)+1$ and is non-singular when $r\wedge d=1$ but usually singular when $r\wedge d\neq 1$ (unless, in fact, $g=2$, $r=2$ and $d=0$). Finally, when $g\geq 2$, there are always stable bundles of rank $r$ and degree $d$ over $X$ (by the theorem of Narasimhan and Seshadri, \cite{NS}, these come from irreducible rank $r$ unitary representations of a certain central extension of $\pi_1(X)$ by $\Z$, determined by $d$ up to isomorphism). If now $g=1$, then the results of Atiyah (\cite{Atiyah_elliptic_curves}) and Tu (\cite{Tu}) show that the set of $S$-equivalence classes of semi-stable vector bundles of rank $r$ and degree $d$ admits a structure of non-singular complex projective variety of dimension $r\wedge d$ (which is consistent with the formula for $g\geq 2$ only when $r$ and $d$ are coprime). But now stable vector bundles of rank $r$ and degree $d$ can only exist if $r\wedge d=1$, as Tu showed following Atiyah's results (\cite[Theorem A]{Tu}). In particular, the structure of poly-stable vector bundles over a complex elliptic curve is rather special, as recalled next.

\begin{proposition}[Atiyah-Tu]\label{poly-stable_bundles}
Let $\cE$ be a poly-stable holomorphic vector bundle of rank $r$ and degree $d$ over a compact connected Riemann surface $X$ of genus $1$. Let us set $h:=r\wedge d$, $r':=\frac{r}{h}$ and $d':=\frac{d}{h}$. Then $\cE \simeq \cE_1\oplus\cdots\oplus \cE_h$ where each $\cE_i$ is a stable holomorphic vector bundle of rank $r'$ and degree $d'$.
\end{proposition}

\begin{proof}
By definition, a poly-stable bundle of rank $r$ and degree $d$ is isomorphic to a 
direct sum $\cE_1\oplus \cdots\oplus \cE_k$ of stable bundles of slope 
$\frac{d}{r}=\frac{d'}{r'}$. Since $d'\wedge r'=1$ and each $\cE_i$ is stable of 
slope $\frac{d'}{r'}$, each $\cE_i$ must have rank $r'$ and degree $d'$ (because 
stable bundles over elliptic curves must have coprime rank and degree). Since 
$\mathrm{rk}(\cE_1\oplus\cdots\oplus\cE_k) = kr'=\rk(\cE)=r$, we have indeed $k=h$.
\end{proof}

To understand the moduli space $\ModC$ of semi-stable holomorphic vector bundles of rank $r$ and degree $d$ over a complex elliptic curve $X$, one then has the next two theorems.

\begin{theorem}[Atiyah-Tu]\label{moduli_space_coprime_case}
Let $X$ be a compact connected Riemann surface of genus $1$ and assume that $r\wedge d=1$. Then the determinant map $\det:\ModC\lra\Pic^{\,d}_X$ is an isomorphism of complex analytic manifolds of dimension $1$.
\end{theorem}

Note that, when $r\wedge d=1$, any semi-stable vector bundle of rank $r$ and degree $d$ is in fact stable (over a curve of arbitrary genus) and that, to prove Theorem \ref{moduli_space_coprime_case}, it is in particular necessary to show that a stable vector bundle $\cE$ of rank $r$ and degree $d$ over a complex elliptic curve $X$ satisfies $\cE\otimes \cL\simeq \cE$ if and only if $\cL$ is an $r$-torsion point in $\Pic^{\,0}_X$ (i.e., $\cL^{\otimes r}\simeq O_X$), a phenomenon which only occurs in genus $1$.

If now $h:=r\wedge d\geq 2$, then we know, by Proposition \ref{poly-stable_bundles}, that a semi-stable vector bundle of rank $r$ and degree $d$ is isomorphic to the direct sum of $h$ stable vector bundles of rank $r'=r/h$ and degree $d'=d/h$. Combining this with Theorem \ref{moduli_space_coprime_case}, one obtains the following result, due to Tu.

\begin{theorem}[{\cite[Theorem 1]{Tu}}]\label{moduli_space_general_case}
Let $X$ be a compact connected Riemann surface of genus $1$ and denote by
$h:=r\wedge d$. Then there is an isomorphism of complex analytic manifolds
$$\begin{array}{ccc}
\ModC & \overset{\simeq}{\lra} & \Sym^h(\Pic^{d/h}_X)\\
\cE \simeq \cE_1\oplus \cdots \oplus \cE_h & \lmt & [\det\cE_i]_{1\leq i\leq h}
\end{array}.$$
\end{theorem}

\noindent In particular, $\ModC$ has dimension $h=r\wedge d$. Since the choice of a point $x_0\in X$ provides an isomorphism $\Pic^{\,d}_X\simeq_\C X$, we have indeed $\ModC \simeq_\C \Sym^h(X)$. In the next section, we will analyze the corresponding situation over $\R$. But first we recall the basics about real and quaternionic vector bundles.

Let $(X,\si)$ be a real Riemann surface, i.e., a Riemann surface $X$ endowed with a real structure $\si$. A real holomorphic vector bundle over $(X,\si)$ is a pair $(\cE,\tau)$ such that $\cE\lra X$ is a holomorphic vector bundle over $X$ and $\tau:\cE\lra\cE$ is an anti-holomorphic map such that

\begin{enumerate}
\item the diagram
$$\begin{CD}
\cE_1 @>{\tau}>> \cE_2\\
@VVV @VVV\\
X @>{\si}>> X
\end{CD}$$ is commutative;
\item the map $\tau$ is fibrewise $\C$-anti-linear: $\forall v\in\cE$, $\forall \lambda\in\C$, $\tau(\lambda v) =\ov{\lambda}\tau(v)$;
\item $\tau^2=\mathrm{Id}_\cE$.
\end{enumerate}

\noindent A quaternionic holomorphic vector bundle over $(X,\si)$ is a pair $(\cE,\tau)$ satisfying Conditions (1) and (2) above, as well as a modified third condition: (3)' $\tau^2=-\mathrm{Id}_\cE$. A homomorphism $\phi:(\cE_1,\tau_1) \lra (\cE_2,\tau_2)$ between two real (respectively,\ quaternionic) vector bundles is a holomorphic map $\phi:\cE_1\lra\cE_2$ such that

\begin{enumerate}
\item the diagram
$$\begin{CD}
\cE_1 @>{\phi}>> \cE_2 \\
@VVV @VVV \\
X @= X 
\end{CD}$$ is commutative;
\item $\phi\circ\tau_1=\tau_2\circ\phi$.
\end{enumerate}

A real (respectively,\ quaternionic) holomorphic vector bundle is called stable if for any $\tau$-invariant sub-bundle $\cF\subset \cE$, one has $\mu(\cF)<\mu(\cE)$. It is called semi-stable if for any such $\cF$, one has $\mu(\cF)\leq\mu(\cF)$. As shown in \cite{Sch_JSG}, $(\cE,\tau)$ is semi-stable as a real (respectively,\ quaternionic) vector bundle if and only $\cE$ is semi-stable as a holomorphic vector bundle but $(\cE,\tau)$ may be stable as a real (respectively,\ quaternionic) vector bundle while being only poly-stable as a holomorphic vector bundle (when $\cE$ is in fact stable, we will say that $(\cE,\tau)$ is geometrically stable). However, any semi-stable
real (respectively,\ quaternionic) vector bundle admits real (respectively,\ quaternionic) Jordan-H\"older filtrations (where the successive quotients can sometimes be stable in the real sense only) and there is a corresponding notion of poly-stable
real (respectively,\ quaternionic) vector bundle, which turns out to be equivalent to being poly-stable and real (respectively,\ quaternionic). Real and quaternionic vector bundles over a compact connected real Riemann surface $(X,\si)$ were topologically classified in \cite{BHH}. If $X^\si\neq\emptyset$, a real vector bundle $(\cE,\tau)$ over $(X,\si)$ defines in particular a real vector bundle in the ordinary sense $\cE^\tau\lra X^\si$, hence an associated first Stiefel-Whitney class $w_1(\cE^\tau)\in H^1(X^\si;\Z/2\Z)\simeq (\Z/2\Z)^n$, where $n\in\{0,\cdots,g+1\}$ is the
number of connected components of $X^\si$. The topological classification of real and quaternionic vector bundles then goes as follows.
\begin{theorem}[\cite{BHH}]\label{top_classif}
Let $(X,\si)$ be a compact connected real Riemann surface.
\begin{enumerate}
\item If $X^\si\neq\emptyset$, real vector bundles over $(X,\si)$ are classified up to smooth isomorphism by the numbers $r=\rk(\cE)$, $d=\deg(\cE)$ and $(s_1,\cdots,s_n)=w_1(\cE^\tau)$, subject to the condition $s_1+\ldots+s_n=d\ \mathrm{mod}\ 2$. Quaternionic vector bundles must have even rank and degree in this case and are classified up to smooth isomorphism by the pair $(2r,2d)$.
\item If $X^\si=\emptyset$, real vector bundles over $(X,\si)$ must have even degree are classified up to smooth isomorphism by the pair $(r,2d)$. Quaternionic vector bundles are classified up to smooth isomorphism by the pair $(r,d)$, subject to the condition $d+r(g-1)\equiv 0\ (\mathrm{mod}\ 2)$. In particular, if $g=1$, real and quaternionic vector bundles alike must have even degree.
\end{enumerate}
\end{theorem}

\noindent Theorem \ref{top_classif} will be useful in Section \ref{topology}, for the proof of Theorem \ref{real_pts_coprime_case}.

\subsection{The real structure of the moduli space}\label{real_structure_on_mod_space}

Let first $(X,\si)$ be a real Riemann surface of arbitrary genus $g$. Then the involution $\cE\lmt\os{\cE}$ preserves the rank and the degree of a holomorphic vector bundle and the bundle $\os{\cE}$ is stable (respectively,\ semi-stable) if and only if $\cE$ is. Moreover, if $\cE$ is semi-stable, a Jordan-H\"older filtration of $\cE$ is mapped to a Jordan-H\"older filtration of $\os{\cE}$, so, for any $g$, the moduli space $\ModC$ of semi-stable holomorphic vector bundles of rank $r$ and degree $d$ on $X$ has an induced real structure. Assume now that $g=1$ and let us prove Theorem \ref{moduli_space_over_R}.

\begin{proof}[Proof of Theorem \ref{moduli_space_over_R}] Since, for any vector bundle $\cE$ one has $\det(\os{\cE})=\os{(\det\cE)}$, the map $$\begin{array}{ccc}
\ModC & \overset{\simeq}{\lra} & \Sym^h(\Pic^{d/h}_X)\\
\cE \simeq \cE_1\oplus \cdots \oplus \cE_h & \lmt & [\det\cE_i]_{1\leq i\leq h}
\end{array}$$ of Theorem \ref{moduli_space_general_case} is a real map. If $X^\si\neq\emptyset$, then by Theorem \ref{line_bdle_case}, we have $$\Pic^{\,d}_X\simeq_\R\Pic^{\,0}_X\simeq_\R X$$ so $\ModC \simeq_\R \Sym^h(X)$ in this case. And if $X^\si=\emptyset$, we distinguish between the cases $d=2e+1$ and $d=2e$ to obtain, again by Theorem \ref{line_bdle_case}, that $$\ModC\simeq_\R \left\{ \begin{array}{ccl}
\Sym^h(X) & \mathrm{if} & d/h\ \mathrm{is\ odd},\\
\Sym^h(\Pic^{\,0}_X) & \mathrm{if} & d/h\ \mathrm{is\ even},
\end{array}\right.$$ which finishes the proof of Theorem \ref{moduli_space_over_R}.
\end{proof}

Let us now focus on the case $d=0$, where there is a nice alternate description of the 
moduli variety in terms of representations of the fundamental group of the elliptic 
curve $(X,x_0)$. Since $\pi_1(X,x_0)\simeq \Z^2$ is a free Abelian group on two 
generators, a rank $r$ unitary representation of it is entirely determined by the data 
of two commuting unitary matrices $u_1,u_2$ in $\mathbf{U}(r)$ (in particular, such a 
representation is never irreducible unless $r=1$) and we may assume that these two 
matrices lie in the maximal torus $\T_r\subset \mathbf{U}(r)$ consisting of diagonal 
unitary matrices. The Weyl group of $\T_r$ is $\W_r\simeq\mathfrak{S}_r$, the symmetric group on 
$r$ letters, and one has 
\begin{equation}\label{torus_reduction}\Hom(\pi_1(X,x_0);\U(r))/\U(r) \simeq 
\Hom(\pi_1(X,x_0);\T_r)/\W_r.\end{equation} Note that since $\pi_1(X,x_0)$ is Abelian, 
there is a well-defined action of $\si$ on it even if $x_0\notin X^\si$: a loop $\gamma$ 
at $x_0$ is sent to the loop $\si\circ\gamma$ at $\si(x_0)$ then brought back to $x_0$ 
by conjugation by an arbitrary path between $x_0$ and $\si(x_0)$. Combining this with 
the involution $u\lmt\ov{u}$ of $\U(r)$, we obtain an action of $\si$ on 
$\Hom(\pi_1(X,x_0);\U(r))$, defined by sending a representation $\rho$ to the representation 
$\si\rho\si$. This action preserves the subset $\Hom(\pi_1(X,x_0);\T_r)$ and is 
compatible with the conjugacy action of $\U(r)$ in the sense that 
$\si(\mathrm{Ad}_u\,\rho)\si=\mathrm{Ad}_{\si(u)}\, (\si\rho\si^{-1})$, so it induces an 
involution on the representation varieties $\Hom(\pi_1(X,x_0);\U(r))/\U(r)$ and 
$\Hom(\pi_1(X,x_0);\T_r)/\W_r$ and the bijection \eqref{torus_reduction} is equivariant 
for the actions just described. By the results of Friedman, Morgan and Witten in 
\cite{FMW} and Laszlo in \cite{Laszlo}, this representation variety is in fact 
isomorphic to the moduli space $\cM_X(r,0)$. Moreover, the involution $\cE\lmt\os{\cE}$ 
on bundles correspond to the involution $\rho\lmt\si\rho\si$ on unitary representations. 
Moreover, $$\T_r \simeq \underbrace{\U(1)\times\cdots\times\U(1)}_{r\ \mathrm{times}} 
\simeq \U(1)\otimes_{\Z}\Z^r$$ as Abelian Lie groups, where $\Z^r$ can be interpreted as 
$\pi_1(\T_r)$. In particular, the Galois action induced on $\Z^r$ by the complex 
conjugation on $\T_r$ is simply $(n_1,\cdots,n_r)\lmt(-n_1,\cdots,-n_r)$ and the 
isomorphism $\T_r\simeq\U(1)\otimes\Z^r$ is equivariant with respect to these natural 
real structures. Finally, the bijection $$\Hom(\pi_1(X,x_0);\T_r) \simeq 
\Hom(\pi_1(X,x_0);\U(1))\otimes\Z^r$$ is also equivariant and the representation variety 
$\Hom(\pi_1(X,x_0);\U(1))$ is isomorphic to $\Pic^{\,0}_X$ as a real variety. We have 
thus proved the following result, which is an analogue over $\R$ of one of the results 
in \cite{FMW,Laszlo}.

\begin{theorem}\label{moduli_space_d_equal_0}
Let $(X,x_0,\si)$ be a real elliptic curve. Then the map $$\begin{array}{ccc}
(\Pic^{\,0}_X\otimes_{\Z}\Z^r) \simeq \big(\Pic^{\,0}_X)^r & \lra & \cM(r,0)\\
(\cL_1,\cdots,\cL_r) & \lmt & \cL_1\oplus\cdots\oplus\cL_r
\end{array}$$ induces an isomorphism $$\cM_X(r,0)\simeq_{\R} (\Pic^{\,0}_X\otimes \Z^r)/\mathfrak{S}_r,$$ where the symmetric group $\mathfrak{S}_r$ acts on $\Z^r$ by permutation.\\ When $X^\si\neq\emptyset$, one can further identify $\Pic^{\,0}_X$ with $X$ over $\R$ and obtain the isomorphism $$\cM_X(r,0) \simeq_{\R} (X\otimes\Z^r) / \mathfrak{S}_r.$$
\end{theorem}

The results of Section \ref{indecomp_bdles} will actually give an alternate proof of Theorem \ref{moduli_space_d_equal_0}, by using the theory of indecomposable vector bundles over elliptic curves. We point out that algebraic varieties of the form $(X\otimes\pi_1(\T))/\W_{\T}$ for $X$ a complex elliptic curve have been studied for instance by Looijenga in \cite{Looijenga}, who identified them with certain weighted projective spaces determined by the root system of $\T$, when the ambient group $G\supset \T$ is semi-simple. Theorem \ref{moduli_space_d_equal_0} shows that, over $\R$, it may sometimes be necessary to replace $X$ by $\Pic^0_X$.

To conclude on the case where $d=0$, we recall that, on $\cM_X(r,0)$, there exists another real structure, obtained from the real structure $\cE\lmt\os{\cE}$ by composing it with the holomorphic involution $\cE\lmt \cE^*$, which in general sends a vector bundle of
degree $d$ to a vector bundle of degree $-d$, so preserves only the moduli spaces $\cM_X(r,0)$. Denote then by
$$
\eta_r\, :
\begin{array}{ccc} {\mathcal M}_X(r,0) & \longrightarrow & {\mathcal M}_X(r,0)\\
\cE & \longmapsto & \overline{\sigma^*E}^{\,*}
\end{array}
$$ this new real structure on the moduli space ${\mathcal M}_X(r,0)$. In particular, we have
$$
\eta_1\, :
\begin{array}{ccc} \Pic^{\,0}_X & \longrightarrow & \Pic^{\,0}_X\\
\cL & \longmapsto & \overline{\sigma^*\cL}^{\,*}
\end{array}
$$
and we note that $\eta_1$ has real points because it fixes the trivial line bundle. The real elliptic curve $(\Pic^{\,0}_X,\eta_1)$ is, in general, not isomorphic to $(X,\sigma)$, even when $\sigma$ has fixed points. We can nonetheless characterize the new real structure of the moduli spaces $\cM_X(r,0)$ in the following way.

\begin{proposition}\label{prop1}
The real variety $({\mathcal M}_X(r,0), \eta_r)$ is isomorphic to the $r$-fold
symmetric product of the real elliptic curve $(\Pic^{\,0}_X,\eta_1)$.
\end{proposition}

\begin{proof}
The proposition is proved in the same way as Theorem
\ref{moduli_space_d_equal_0}, changing only the real structures under consideration.
\end{proof}

\subsection{Topology of the set of real points in the coprime case}\label{topology}

In rank $1$, the topology of the set of real points of $\Pic^{\,d}_X$ is well understood and so is the modular interpretation of its elements.

\begin{theorem}[\cite{GH}, case $g=1$]\label{GH_case_g_equal_1}
Let $(X,\si)$ be a compact real Riemann surface of genus $1$ and let $d\in\Z$.
\begin{enumerate}
\item If $X^\si\neq\emptyset$, then $(\Pic^{\,d}_X)^\si\simeq X^\si$ has $1$ or $2$ connected components. Elements of $(\Pic^{\,d}_X)^\si$ correspond to real isomorphism classes of real holomorphic line bundles over $(X,\si)$ and two such real line bundles $(\cL_1,\tau_1)$ and $(\cL_2,\tau_2)$ lie in the same connected component of $(\Pic^{\,d}_X)^\si$ if and only if $w_1(\cL_1^{\tau_1})= w_1(\cL_2^{\tau_2})$.
\item If $X^\si=\emptyset$ and $d=2e+1$, then $(\Pic^{\,d}_X)^\si\simeq X^\si$ is empty.
\item If $X^\si=\emptyset$ and $d=2e$, then $(\Pic^{\,d}_X)^\si\simeq (\Pic^{\,0}_X)^\si$ has $2$ connected components, corresponding to isomorphism classes of either real or quaternionic line bundles of degree $d$, depending on the connected component of $(\Pic^{\,d}_X)^\si$ in which they lie. 
\end{enumerate}
Moreover, in cases $\textstyle{(1)}$ and $\textstyle{(3)}$, any given connected component of $(\Pic^{\,d}_X)^\si$ is diffeomorphic to $S^1$.
\end{theorem}

For real Riemann surfaces of genus $g\geq 2$, the topology of $(\Pic^{\,d}_X)^\si$, in particular the number of connected components, is a bit more involved but also covered in \cite{GH}, the point being that these components are indexed by the possible topological types of real and quaternionic line bundles over $(X,\si)$. For vector bundles of rank $r\geq 2$ on real Riemann surfaces of genus $g\geq 2$, a generalization of the results of Gross and Harris was obtained in \cite{Sch_JSG}: we recall here the result for coprime rank and degree (in general, a similar but more complicated result holds provided one restricts one's attention to the stable locus in $\ModC$). The coprime case is the case that we will actually generalize to genus $1$ curves (where stable bundles can only exist in coprime rank and degree).

\begin{theorem}[\cite{Sch_JSG}]\label{top_real_pts_hyp_case}
Let $(X,\si)$ be a compact real Riemann surface of genus $g\geq 2$ and assume that $r\wedge d=1$. The number of connected component of $\ModC^\si$ is equal to:
\begin{enumerate}
\item $2^{n-1}$ if $X^\si$ has $n>0$ connected components. In this case, elements of $\ModC^\si$ correspond to real isomorphism classes of real holomorphic vector bundles of rank $r$ and degree $d$ and two such bundles $(\cE_1,\tau_1)$ and $(\cE_2,\tau_2)$ lie in the same connected component if and only if $w_1(\cE^{\tau_1})=w_1(\cE_2^{\tau_2})$.
\item $0$ if $X^\si=\emptyset$, $d$ is odd and $r(g-1)$ is even.
\item $1$ if $X^\si=\emptyset$, $d$ is odd and $r(g-1)$ is odd, in which case the elements of $\ModC^\si$ correspond to quaternionic isomorphism classes of quaternionic vector bundles of rank $r$ and degree $d$.
\item $1$ if $X^\si=\emptyset$, $d$ is even and $r(g-1)$ is odd, in which case the elements of $\ModC^\si$ correspond to real isomorphism classes of real vector bundles of rank $r$ and degree $d$.
\item $2$ if $X^\si=\emptyset$, $d$ is even and $r(g-1)$ is even, in which case there is one component consisting of real isomorphism classes of real vector bundles of rank $r$ and degree $d$ while the other consists of quaternionic isomorphism classes of quaternionic vector bundles of rank $r$ and degree $d$.
\end{enumerate}
\end{theorem}

Now, using Theorem \ref{moduli_space_over_R}, we can extend Theorem \ref{top_real_pts_hyp_case} to the case $g=1$. Indeed, to prove Theorem \ref{real_pts_coprime_case}, we only need to combine Theorem \ref{GH_case_g_equal_1} and the coprime case of Theorem \ref{moduli_space_over_R} (i.e.,\ $h=1$), with the following result, for a proof of which we refer to either \cite{BHH} or \cite{Sch_JSG}.

\begin{proposition}\label{self_conj_stable_bundles}
Let $(X,\si)$ be a compact connected real Riemann surface and let $\cE$ be a stable holomorphic vector bundle over $X$ satisfying $\os{\cE}\simeq\cE$. Then $\cE$ is either real or quaternionic and cannot be both. Moreover, two different real or quaternionic structures on $\cE$ are conjugate by a holomorphic automorphism of $\cE$. 
\end{proposition}

\noindent Note that it is easy to show that two real (respectively,\ quaternionic) structures on $\cE$ \textit{differ} by a holomorphic automorphism $e^{i\theta}\in S^1\subset \C^*\simeq \Aut(\cE)$ but, in order to prove that these two structures $\tau_1$ and $\tau_2$, say, \textit{are conjugate}, we need to observe that $\tau_2(\,\cdot\,)=e^{i\theta}\tau_1(\,\cdot\,)=e^{i\theta/2}\tau_1(e^{-i\theta/2}\,\cdot\,)$. Then, to finish the proof of Theorem \ref{real_pts_coprime_case}, we proceed as follows. 

\begin{proof}[Proof of Theorem \ref{real_pts_coprime_case}]
Recall that $X$ here has genus $1$. If $X^\si\neq\emptyset$, quaternionic vector bundles must have even rank and degree by Theorem \ref{top_classif}, so, by Proposition \ref{self_conj_stable_bundles}, points of $\ModC^\si$ correspond in this case to real isomorphism classes of geometrically stable real vector bundles of rank $r$ and degree $d$. By Theorem \ref{moduli_space_over_R}, one indeed has $\ModC^\si\simeq (\Pic^{\,d}_X)^\si \simeq X^\si$ in this case. Moreover, since the diffeomorphism $\ModC^\si\simeq (\Pic^{\,d}_X)^\si$ is provided by the determinant map, the connected components of $\ModC^\si$, or equivalently of $(\Pic^{\,d}_X)^\si$ are indeed distinguished by the first Stiefel-Whitney class of the real part of the real bundles that they parametrize, as in Theorem \ref{GH_case_g_equal_1}. If now $X^\si=\emptyset$,  then by Theorem \ref{top_classif}, real and quaternionic vector bundles must have even degree and we can again use Theorem \ref{GH_case_g_equal_1} to conclude: note that since the diffeomorphism $\ModC^\si\simeq
(\Pic^{\,d}_X)^\si$ is provided by the determinant map, when $d$ is even $r$ must be
odd (because $r$ is assumed to be coprime to $d$), so the determinant indeed takes real vector bundles to real line bundles and quaternionic vector bundles to quaternionic line bundles.
\end{proof}

Had we not assumed $r\wedge d=1$, then the situation would have been more complicated to analyze, because the determinant of a quaternionic vector bundle of even rank is a real line bundle and also because, when $h=r\wedge d$ is even, the real space $\Sym^h(X)$ may have real points even if $X$ does not (points of the form $[x_i,\si(x_i)]_{1\leq i\leq \frac{h}{2}}$ for $x_i\in X$).

\subsection{Real vector bundles of fixed determinant}\label{fixed_det_case_section}

Let us now consider spaces of vector bundles of fixed determinant. By Theorem 3 of
\cite{Tu}, one has, for any $\cL\in\Pic^{\,d}_X$, $\cM_X(r,\cL):=\det ^{-1}(\{\cL\})
\simeq_{\C} \mathbb{P}_{\C}^{h-1}$ where $d=\deg(\cL)$ and $h=r\wedge d$. This is proved
in the following way: under the identification $\ModC\simeq_{\C}\Sym^h(X)$, there is
a commutative diagram $$\begin{CD}
\ModC @>{\simeq}>> \Sym^h(X)\\
@VV{\det}V @VV{\mathrm{AJ}}V \\
\Pic^{\,d}_X @>{\simeq}>> \Pic^{\,h}_X
\end{CD}$$ where $$\mathrm{AJ}: \begin{array}{ccc}
\Sym^h(X) & \lra & \Pic^{\,h}_X \\
(x_1,\cdots,x_h) & \lmt & \cL(x_1+\ldots+x_h)
\end{array}$$ is the Abel-Jacobi map (taking a finite family of points $(x_1,\cdots,x_h)$ to the line bundle associated to the divisor $x_1+\ldots+x_h$) and the fiber of the Abel-Jacobi map above a holomorphic line bundle $\cL$ of degree $h$ is the projective space $\mathbb{P}(H^0(X,\cL))$ which, since $\deg(\cL)=h\geq 1$ and $X$ has genus $1$, is isomorphic to $\mathbb{P}_{\C}^{h-1}$. Evidently, the same proof will work over $\R$ whenever we can identify $\Pic^{\,d}_X$ and $\Pic^{\,h}_X$ as real varieties, which happens in particular when $X^\si\neq\emptyset$.

\begin{theorem}\label{fixed_det_case}
Let $(X,x_0,\si)$ be a real elliptic curve satisfying $X^\si\neq\emptyset$ and let $\cL$ be a real line bundle of degree $d$ on $X$. Then, for all $r\geq 1$, $$\cM_X(r,\cL) \simeq_{\R} \mathbb{P}_{\R}^{h-1}$$ where $h=r\wedge d$.
\end{theorem}

\begin{proof}
When $X^\si\neq\emptyset$, we can choose $x_0\in X^\si$ and use Theorem \ref{line_bdle_case} to identify all Picard varieties $\Pic^{\,k}_X$ over $\R$, then reproduce Tu's proof recalled above.
\end{proof}

\section{Indecomposable vector bundles}\label{indecomp_bdles}

\subsection{Indecomposable vector bundles over a complex elliptic curve}

As recalled in the introduction, a theorem of Grothendieck of 1957 shows that any holomorphic vector bundle on $\CP^1$ is isomorphic to a direct sum of holomorphic line bundles (\cite{Grot_P1}) and this can be easily recast in modern perspective by using the notions of stability and semi-stability of vector bundles over curves, introduced by Mumford in 1962 and first studied by himself and Seshadri (\cite{Mumford_Proc,Seshadri}): the moduli variety of semi-stable vector bundles of rank $r$ and degree $d$ over $\CP^1$ is a single point if $r$ divides $d$ and is empty otherwise. As for vector bundles over a complex elliptic curve, the study was initiated by Atiyah in a paper published in 1957, thus at a time when the notion of stability was not yet available. Rather, Atiyah's starting point in \cite{Atiyah_elliptic_curves} is the notion of an indecomposable vector bundle: a holomorphic vector bundle $\cE$ over a complex curve $X$ is said to be indecomposable if it is not isomorphic to a direct sum of non-trivial holomorphic bundles. In the present paper, we shall denote by $\cI_X(r,d)$ the set of isomorphism classes of indecomposable vector bundles of rank $r$ and degree $d$. It is immediate from the definition that a holomorphic vector bundle is a direct sum of indecomposable ones. Moreover, one has the following result, which is a consequence of the categorical Krull-Schmidt theorem, also due to Atiyah (in 1956), showing that the decomposition of a bundle into indecomposable ones is essentially unique.

\begin{proposition}[{\cite[Theorem 3]{Atiyah_Krull-Schmidt}}]\label{Atiyah_Krull-Schmidt}
Let $X$ be a compact connected complex analytic manifold. Any holomorphic vector bundle $\cE$ over $X$ is isomorphic to a direct sum $\cE_1\oplus\cdots\oplus\cE_k$ of indecomposable vector bundles. If one also has $\cE\simeq\cE'_1\oplus\cdots\oplus\cE'_{l}$, then $l=k$ and there exists a permutation $\si$ of the indices such that $\cE'_{\si(i)}\simeq\cE_i$.
\end{proposition}

Going back to the case of a compact, connected Riemann surface $X$ of genus $1$, Atiyah completely describes all indecomposable vector bundles on $X$. He first shows the existence, for any $h\geq 1$, of a unique (isomorphism class of) indecomposable vector bundle $F_h$ of rank $h$ and degree $0$ such that \begin{equation}\label{dim_space_of_sections_of_F_h}
\dim H^0(X;F_h)=1
\end{equation} (\cite[Theorem 5]{Atiyah_elliptic_curves}). As a matter of fact, this is the only vector bundle of rank $h$ and degree $0$ over $X$ with a non-zero space of sections. Let us call $F_h$ the \textit{Atiyah bundle} of rank $h$ and degree $0$. The construction of $F_h$ is by induction, starting from $F_1=\cO_X$, the trivial line bundle over $X$, and showing the existence and uniqueness of an extension of the form \begin{equation}\label{ses_defining_F_h}
0 \lra \cO_X \lra F_h \lra F_{h-1} \lra 0.
\end{equation} In particular, $\det(F_h)=\cO_X$. Moreover, Since $F_h$ is the unique indecomposable vector  bundle with non-zero space of sections, one has (\cite[Corollary 1]{Atiyah_elliptic_curves}):
 \begin{equation}\label{F_h_self-dual}
F_h^*\simeq F_h.
\end{equation} Note that $F_h$ is an extension of semi-stable bundles so it is semi-stable. The associated poly-stable bundle is the trivial bundle of rank $h$, which is not isomorphic to $F_h$ (in particular, $F_h$ is not itself poly-stable).

Atiyah then shows that any indecomposable vector bundle $\cE$ of rank $h$ and degree $0$ is isomorphic to $F_h\otimes \cL$ for a line bundle $\cL$ of degree $0$ which is unique up to isomorphism (\cite[Theorem 5-(ii)]{Atiyah_elliptic_curves}). Since it follows from the construction of $F_h$ recalled in \eqref{ses_defining_F_h} that $\det(F_h)=\cO_X$, one has $\det\cE=\cL^h$. This sets up a bijection \begin{equation}\label{isom_with_line_bundles}
\begin{array}{ccc}
\Pic^{\,0}_X & \lra & \cI_X(h,0)\\
\cL & \lmt & F_h\otimes\cL
\end{array}.
\end{equation} Note that the map \eqref{isom_with_line_bundles} is just the identity map if $h=1$. Atiyah then uses a marked point $x_0\in X$ to further identify $\Pic^{\,0}_X$ with $X$. In particular, the set $\cI_X(h,0)$ inherits a natural structure of complex analytic manifold of dimension $1$.

The next step in Atiyah's characterization of indecomposable vector bundles is to consider the case of vector bundles of non-vanishing degree. He shows that, associated to the choice of a marked point $x_0\in X$, there is, for all $r$ and $d$, a unique bijection (subject to certain conditions) \begin{equation}\label{Atiyah_map}\alpha^{\,x_0}_{r,d}:\mathcal{I}_X(r\wedge d,0) \lra \mathcal{I}_X(r,d)\end{equation} between the set of isomorphism classes of indecomposable vector bundles of rank $h:=r\wedge d$ (the g.c.d.\ of $r$ and $d$) and degree $0$ and the set of isomorphism classes of indecomposable vector bundles of rank $r$ and degree $d$ (\cite[Theorem 6]{Atiyah_elliptic_curves}). As a consequence, Atiyah can define a canonical indecomposable vector bundle of rank $r$ and degree $d$, namely \begin{equation*}F_{x_0}(r,d):=\alpha^{\,x_0}_{r,d}(F_{r\wedge d})\end{equation*} where $F_{r\wedge d}$ is the indecomposable vector bundle of rank $r\wedge d$ and degree $0$ whose construction was recalled in \eqref{ses_defining_F_h}. We will call the bundle $F_{x_0}(r,d)$ the Atiyah bundle of rank $r$ and degree $d$ (in particular $F_{x_0}(r,0) = F_r$). Atiyah then obtains the following description of indecomposable vector bundles.

\begin{theorem}[{\cite[Theorem 10]{Atiyah_elliptic_curves}}]\label{Atiyah_indecomp_bdles}
Let us set $h=r\wedge d$, $r'=\frac{r}{h}$ and $d'=\frac{d}{h}$. Then every indecomposable vector bundle of rank $r$ and degree $d$ is isomorphic to a bundle of the form $F_{x_0}(r,d)\otimes \cL$ where $\cL$ is a line bundle of degree $0$. Moreover, $F_{x_0}(r,d)\otimes \cL\simeq F_{x_0}(r,d)\otimes\cL'$ if and only if $(\cL'\otimes\cL^{-1})^{r'}\simeq\cO_X$.
\end{theorem}

Thus, as a generalization to \eqref{isom_with_line_bundles}, Theorem \ref{Atiyah_indecomp_bdles} shows that there is a surjective map $\Pic^{\,0}_X\lra \cI_X(r,d)$, whose fiber is isomorphic to the group $T_{r'}$ of $r'$-torsion elements in $\Pic^{\,0}_X$. This in particular induces a bijection between the Riemann surface $\Pic^{\,0}_X/T_{r'}\simeq\Pic^{\,0}_X\simeq X$ and the set $\cI_X(r,d)$ for all $r$ and $d$ and the set $\cI_X(r,d)$ inherits in this way a natural structure of complex analytic manifold of dimension $1$.

\subsection{Relation to semi-stable and stable bundles}

It is immediate to prove that stable bundles (over a curve of arbitrary genus) are indecomposable. Moreover, over an elliptic curve, we have the following result, for a proof of which we refer to Tu's paper.

\begin{theorem}[{\cite[Appendix A]{Tu}}]\label{Tu_obs}
Every indecomposable vector bundle of rank $r$ and degree $d$ over a complex elliptic curve is semi-stable. It is stable if and only $r\wedge d=1$.
\end{theorem}

\noindent In particular, the Atiyah bundles $F_{x_0}(r,d)$ are semi-stable (and stable if and only if $r\wedge d=1$) and, by Proposition \ref{Atiyah_Krull-Schmidt}, every holomorphic vector bundle over a complex elliptic curve is isomorphic to a direct sum of semi-stable bundles. Next, there is a very important relation between indecomposable vector bundles and stable vector bundles, which will be useful in the next section.

\begin{theorem}[Atiyah-Tu]\label{rel_between_stable_and_indecomp}
Set $h=r\wedge d$, $r'=\frac{r}{h}$ and $d'=\frac{d}{h}$. Then the map
$$\begin{array}{ccc}
\mathcal{M}_X(r',d') & \lra & \cI_X(r,d)\\
\cE' & \lmt & \cE'\otimes F_h
\end{array}$$ is a bijection: any indecomposable vector bundle of rank $r$ and degree $d$ is isomorphic to a bundle of the form $\cE'\otimes F_h$ where $\cE'$ is a stable vector bundle of rank $r'$ and degree $d'$, unique up to isomorphism, and $F_h$ is the Atiyah bundle of rank $h$ and degree $0$.
\end{theorem}

\noindent In particular, $\cI_X(r,d)$ inherits in this way a structure of complex analytic manifold of dimension $r'\wedge d'=1$. This result, which generalizes \eqref{isom_with_line_bundles} in a different direction than Theorem \ref{Atiyah_indecomp_bdles}, can be deduced from Atiyah and Tu's papers but we give a proof below for the sake of completeness. It is based on the following lemma.

\begin{lemma}[{\cite[Lemma 24]{Atiyah_elliptic_curves}}]\label{rel_between_the_Atiyah_bdles}
The Atiyah bundles $F_{x_0}(r,d)$ and $F_{x_0}(r',d')$ are related in the following way: $$F_{x_0}(r,d) \simeq F_{x_0}(r',d') \otimes F_h.$$
\end{lemma} 

\begin{proof}[Proof of Theorem \ref{rel_between_stable_and_indecomp}]
Let $\cE\in \cI_X(r,d)$. By Theorem \ref{Atiyah_indecomp_bdles}, there exists a line bundle $\cL$ of degree $0$ such that $\cE\simeq F_{x_0}(r,d)\otimes\cL$. By Lemma \ref{rel_between_the_Atiyah_bdles}, $F_{x_0}(r,d)\simeq F_{x_0}(r',d')\otimes F_h$. Since $r'\wedge d'=1$, Theorem \ref{Tu_obs} shows that $F_{x_0}(r',d')$, hence also $\cE':=F_{x_0}(r',d')\otimes \cL$, are stable bundles of rank $r'$ and degree $d'$. And one has indeed $\cE\simeq \cE'\otimes F_h$. Let now $\cE'$ and $\cE''$ be two stable bundles of rank $r'$ and degree $d'$ such that $\cE'\otimes F_h \simeq \cE''\otimes F_h$. Since stable bundles are indecomposable, Theorem \ref{Atiyah_indecomp_bdles} shows the existence of two line bundles $\cL'$ and $\cL''$ of degree $0$ such that $\cE'\simeq F_{x_0}(r',d')\otimes \cL'$ and $\cE'' \simeq F_{x_0}(r',d')\otimes \cL''$. Tensoring by $F_h$ and applying Lemma \ref{rel_between_the_Atiyah_bdles}, we obtain that $F_{x_0}(r,d) \otimes \cL' \simeq F_{x_0}(r,d) \otimes \cL''$ which, again by Theorem \ref{Atiyah_indecomp_bdles}, implies that $\cL'$ and $\cL''$ differ by an $r'$-torsion point of $\Pic^{\,0}_X$. But then a final application of Theorem \ref{Atiyah_indecomp_bdles} shows that $F_{x_0}(r',d')\otimes \cL' \simeq F_{x_0}(r',d') \otimes\cL''$,
i.e.,\ $\cE'\simeq\cE''$.
\end{proof}

Thus, the complex variety $\cI_X(r,d)\simeq\mathcal{M}_X(r',d')\simeq X$ is a $1$-dimensional sub-variety of the $h$-dimensional moduli variety $\ModC\simeq \Sym^h(X)$ and these two non-singular varieties coincide exactly when $r$ and $d$ are coprime. More explicitly, under the identifications $\cI_X(r,d)\simeq X$ and $\ModC\simeq\Sym^h(X)$, the inclusion map $\cI_X(r,d)\hookrightarrow \ModC$, implicit in Theorem \ref{Tu_obs}, is simply the diagonal map \begin{eqnarray*}
X & \lra & \Sym^h(X)\\
x & \lmt & [x,\cdots,x]
\end{eqnarray*} and it commutes to the determinant map. (the latter being, on $\Sym^h(X)$, just the Abel-Jacobi map $[x_1,\cdots,x_h]\longmapsto x_1+\cdots+x_h$; see \cite[Theorem 2]{Tu}).

\subsection{Indecomposable vector bundles over a real elliptic curve}\label{indecomposable_bdles_over_real_elliptic_curves}

Over a real elliptic curve, the description of indecomposable vector bundles is a bit more complicated than in the complex case, because the Atiyah map $\alpha^{\,x_0}_{r,d}$ defined in \eqref{Atiyah_map} is not a real map unless the point $x_0$ is a real point, which excludes the case where $\Xsi=\emptyset$. Of course the case $\Xsi\neq\emptyset$ is already very interesting and if we follow Atiyah's paper in that case, then the Atiyah map $\alpha^{\,x_0}_{r,d}$ is a real map and the Atiyah bundles $F_{x_0}(r,d)$ are all real bundles. In particular, the description given by Atiyah of the ring structure of the set of isomorphism class of all vector bundles (namely the way to decompose the tensor product of two Atiyah bundles into a direct sum of Atiyah bundles, see for instance \cite[Appendix A]{Tu} for a concise exposition) directly applies to the sub-ring formed by isomorphism classes real bundles (note that, in contrast, isomorphism classes of quaternionic bundles do not form a ring, as the tensor product of two quaternionic bundles is a real bundle). To obtain a description of indecomposable bundles over a real elliptic curve which holds without assuming that the curve has real points, we need to replace the Atiyah isomorphism $$\alpha^{\,x_0}_{r,d}: \cI_X(r\wedge d,0) \lra \cI_X(r,d)$$ (which cannot be a real map when $\Xsi=\emptyset$) by the isomorphism $\cI_X(r,d) \simeq \mathcal{M}_X(r',d')$ of Theorem \ref{rel_between_stable_and_indecomp} and show that the latter is always a real map. The first step is the following result, about the Atiyah bundle of rank $h$ and degree $0$, whose definition was recalled in \eqref{ses_defining_F_h}.

\begin{proposition}\label{F_h_is_real}
Let $(X,\si)$ be a real Riemann surface of genus $1$. For any $h\geq 1$, the indecomposable vector bundle $F_h$ has a canonical real structure.
\end{proposition}

\begin{proof} We proceed by induction. Since $X$ is assumed to be real, $\cO_X$ has a canonical real structure. So, if $h=1$, then $F_h$ is canonically real. Assume now that $h>1$ and that $F_{h-1}$ has a fixed real structure. Following again Atiyah (\cite{Atiyah_analytic_connections}), extensions of the form \eqref{ses_defining_F_h} are parametrized by the sheaf cohomology group $H^1(X;\Hom_{\cO_X}(F_{h-1};\cO_X))=H^1(X;F_{h-1}^{\ *})$. The
uniqueness part of the statement in Atiyah's construction above says that this cohomology group is a complex vector space of dimension $1$, which, in any case, can be checked by Riemann-Roch using Properties \eqref{dim_space_of_sections_of_F_h} and \eqref{F_h_self-dual}. Indeed, since $\deg(F_{h-1}^{\ *})=0$ and $X$ is of genus $g=1$, one has $$h^0(F_{h-1}^{\ *}) - h^1(F_{h-1}^{\ *}) = \deg(F_{h-1}^{\ *}) + (\rk\,F_{h-1}^{\ *})(1-g) =0$$ (where $h^i(\,\cdot\,)=\dim H^i(X;\,\cdot\,)$), so $h^1(F_{h-1}^{\ *})=h^0(F_{h-1}^{\ *})=1$. Now, since $X$ and $F_{h-1}$ have real structures, so does $H^1(X;F_{h-1}^{\ *})$ and the fixed point-space of that real structure corresponds to isomorphism classes of real extensions of $F_{h-1}$ by $\cO_X$. Since the fixed-point space of the real structure of $H^1(X;F_{h-1}^{\ *})$ is a $1$-dimensional real vector space, the real structure of $F_h$ is unique up to isomorphism.
\end{proof}

\noindent Thus, in contrast to Atiyah bundles of non-vanishing degree, $F_h$ is always canonically a real bundle. In particular, $\os{F_h}\simeq F_h$. It is then clear that the isomorphism $$\begin{array}{ccc}
\mathcal{M}_X(r',d') & \lra & \cI_X(r,d)\\
\cE' & \lmt & \cE'\otimes F_h
\end{array}$$ is a real map: $\os{\cE'}\otimes F_h \simeq \os{\cE'} \otimes \os{F_h} \simeq \os{(\cE'\otimes F_h)}$, which readily implies Theorem \ref{indecomp_bdles_over_R}. Moreover, one can make the following observation.

\begin{proposition}\label{self_conj_indecomp_bdles}
Let $\cE$ be an indecomposable vector bundle of rank $r$ and degree $d$ over the real elliptic curve $(X,x_0,\si)$ and assume that $\os{\cE}\simeq \cE$. Then $\cE$ admits either a real or a quaternionic structure.
\end{proposition}

\begin{proof}
By Theorem \ref{rel_between_stable_and_indecomp}, we can write $\cE\simeq \cE'\otimes F_h$, with $\cE'$ stable. Therefore, $$\os{\cE} \simeq \os{(\cE'\otimes F_h)} \simeq \os{\cE'} \otimes \os{F_h} \simeq \os{\cE'} \otimes F_h.$$ The assumption $\os{\cE}\simeq \cE$ then translates to $\os{\cE'}\otimes F_h \simeq \cE'\otimes F_h$ which, since the map from Proposition \ref{rel_between_stable_and_indecomp} is a bijection, shows that $\os{\cE'}\simeq \cE'$. As $\cE'$ is stable, the fact that $\cE'$ admits a real or quaternionic structure $\tau'$ follows from Proposition \ref{self_conj_stable_bundles}. If $\tau_h$ denotes the real structure of $F_h$, we then have that $\tau'\otimes\tau_h$ is a real or quaternionic structure on $\cE$, depending on whether $\tau'$ is real or quaternionic.
\end{proof}

\end{document}